\documentclass[]{amsart}

\usepackage{amscd,amsmath,amssymb,amsfonts}
\usepackage{mathrsfs}
\usepackage[dvips]{color} 
\usepackage[all]{xy}

\numberwithin{equation}{section}

\theoremstyle{plain}
    \newtheorem{thm}{Theorem}[section]
    \newtheorem{lem}[thm]{Lemma}
    
    \newtheorem{prop}[thm]{Proposition}
    
    \newtheorem{prob}[thm]{Problem}
    
    \newtheorem{conj}[thm]{Conjecture}
\theoremstyle{definition}

\def\Coker{\mathrm{Coker}}

\def\dR{{\mathrm{d\hspace{-0.2pt}R}}}            

\def\Image{{\mathrm{Im}}}        
  
\def\Ext{{\mathrm{Ext}}}

\def\ker{{\mathrm{Ker}}}          

\def\End{{\mathrm{End}}}

\def\reg{{\mathrm{reg}}}          %
\def\bA{{\mathbb A}}

\def\C{{\mathbb C}}

\def\P{{\mathbb P}}

\def\Q{{\mathbb Q}}

\def\R{{\mathbb R}}

\def\Z{{\mathbb Z}}
\def\cA{{\mathscr A}}

\def\cM{{\mathscr M}}

\def\a{\alpha}\def\b{\beta}
\def\k{\kappa}\def\z{\zeta}

\def\lra{\longrightarrow}

\def\hra{\hookrightarrow}

\def\ot{\otimes}
\def\op{\oplus}
\def\wt#1{\widetilde{#1}}

\def\ol#1{\overline{#1}}

\def\os#1#2{\overset{#1}{#2}}

\def\Aut{\mathrm{Aut}}

\begin{document}
\title[Special values of hypergeometric functions ${}_3F_2$]
{An algebro-geometric study of special values of hypergeometric functions ${}_3F_2$}
\author[M. Asakura, N. Otsubo and T. Terasoma]{Masanori Asakura, Noriyuki Otsubo and Tomohide Terasoma}
\address{Department of Mathematics, Hokkaido University, Sapporo, 060-0810 Japan}
\email{asakura@math.sci.hokudai.ac.jp}
\address{Department of Mathematics and Informatics, Chiba University, Chiba, 263-8522 Japan}
\email{otsubo@math.s.chiba-u.ac.jp}
\address{Graduate School of Mathematical Sciences, The University of Tokyo, Tokyo, 153-8914 Japan}
\email{terasoma@ms.u-tokyo.ac.jp}
\date{\today}
\dedicatory{To Professor Shuji Saito}

\begin{abstract}

For certain class of hypergeometric functions ${}_3F_2$ with rational parameters, 
we give a sufficient condition for the special value at $1$ to be expressed in terms of logarithms of algebraic numbers. We give two proofs, both of which are algebro-geometric and related to higher regulators. 
\end{abstract}

\maketitle
\section{Introduction}

Special values of hypergeometric functions ${}_pF_q$ are sometimes expressed as an elementary function of their parameters. For example, we have the Euler--Gauss formula
$${}_2F_1\left({a, b\atop c};1\right)=\frac{B(c,c-a-b)}{B(c-a,c-b)}.$$
Here, $B(a,b)=\int_0^1t^{a-1}(1-t)^{b-1}\, dt$ is the beta function. 
In this paper, we study the special values of ${}_3F_2$-functions
\begin{equation}\label{1}
B(a,b)\cdot {}_3F_2\left(
\begin{matrix}
a,b,q\\
a+b,q+1\end{matrix};1
\right)
\end{equation}
for non-integral rational numbers $a,b,q$. 
There is a very classical formula of Watson \cite{watson} (see also \cite[p.98, Example 9]{Bailey}) 
\[
2B(a,b)\cdot {}_3F_2\left(
\begin{matrix}
a,b,\frac{a+b-1}{2}\\
a+b,\frac{a+b+1}{2}\end{matrix};1
\right)=
\psi\left(\frac{a+1}{2}\right)
+\psi\left(\frac{b+1}{2}\right)
-\psi\left(\frac{a}{2}\right)
-\psi\left(\frac{b}{2}\right),
\]
where $\psi(x)=\Gamma'(x)/\Gamma(x)$ is the digamma function. 
In view of Gauss' formula on the values of $\psi(x)$ at rational numbers 
(see \cite[1.7.3, p.18--19]{bateman}), 
Watson's formula implies that, 
when $q=\frac{a+b-1}{2}$, the value \eqref{1} is a $\ol\Q$-linear combination of
of finitely many
$\log\alpha$ with $\alpha\in\Q(\mu_\infty)$.

On the other hand, the recent works \cite{a-o-1}, \cite{a-o} by the first and second authors
show that the value \eqref{1} 
appears as Beilinson's regulator on the motivic cohomology of ``hypergeometric fibrations" (see Theorem \ref{a-o-thm}), 
which is an algebro-geometric invariant related conjecturally with special values of $L$-functions. 
Under a certain geometric assumption concerning the Hodge type, 
the regulator is written in terms of logarithms.
Hence one obtains a sufficient condition for 
\eqref{1} to be written in terms of the logarithms of algebraic numbers, which is the main result of this paper Theorem \ref{main}. 
After the works mentioned above, the third author pointed out that 
the theorem can also be deduced from the study of Fermat surfaces. 
In this paper, we explain both methods, as each one has its advantage and would be useful for future studies.  

The class of $(a,b,q)$ we consider is wider than Watson's formula (see Section 5). 
For example, one shows
\begin{equation*}
2\pi\cdot{}_3F_2\left(
\begin{matrix}\frac{1}{6},\frac{5}{6},\frac{1}{4}\\
1,\frac{5}{4}\end{matrix};1\right)=
\frac{12^{\frac{3}{4}}}{2}\cdot
\log\left(\frac{3^{\frac{5}{4}}-3^{\frac{3}{4}}+\sqrt{2}}{3^{\frac{5}{4}}-3^{\frac{3}{4}}-\sqrt{2}}\right)
-12^{\frac{3}{4}}\cdot \mathrm{Cos}^{-1}\left(\frac{3^{\frac{5}{4}}+3^{\frac{3}{4}}}{2\sqrt{5+3\sqrt{3}}}\right).
\end{equation*}
Here appear, contrary to Watson's formula, the logarithms of non-cyclotomic numbers. 
See also the examples \eqref{ex yabu}.

This paper is organized as follows.
The main theorem is stated in Section \ref{main-sect}.
We give two proofs of the main theorem
 in Section \ref{proof-sect} and Section \ref{fermat-sect}.
The first proof, due to the first and second
authors, uses the regulator of hypergeometric fibrations. 
The second one, due to the third author, uses the regulator of Fermat surfaces.
In Section \ref{problem-sect}, open questions are discussed.

\subsection*{Notations}
Throughout this paper, $\Gamma(s)$ and $B(s,t)$ denote the gamma and beta functions, respectively. The hypergeometric function ${}_3F_2$ is defined by 
$${}_3F_2\left({a,b,c \atop d,e};x\right)=\sum_{n=0}^\infty \frac{(a)_n(b)_n(c)_n}{(d)_n(e)_nn!}x^n, \quad (a)_n=\prod_{i=0}^{n-1}(a+i).$$
It converges as $x \to 1^-$ if and only if $d+e-a-b-c>0$. We write 
$$\Gamma \left({a_1,\dots, a_m \atop b_1,\dots, b_n}\right) = \frac{\prod_{i=1}^m \Gamma(a_i)}{\prod_{j=1}^n \Gamma(b_j)}.$$
For a positive integer $N$, $\mu_N \subset \ol\Q^\times$ denotes the group of $N$th roots of unity.

\section{Main Theorem}\label{main-sect}

For $x \in \Q$, $\{x\}:=x-\lfloor x\rfloor$ denotes the fractional part. 
The map $\{-\}:\Q\to[0,1)$ factors through $\Q/\Z$, which we denote by
the same notation. Let 
$\hat{\Z}=\varprojlim_N\Z/N\Z$ be the profinite completion and
$\hat{\Z}^\times=\varprojlim_N(\Z/N\Z)^\times$ the group of units.
The ring $\hat{\Z}$ acts naturally on the additive group $\Q/\Z$, and induces an isomorphism $\hat{\Z}^\times\cong \Aut(\Q/\Z)$.

Our main theorem is the following. 
\begin{thm}\label{main}
Let $a,b,q\in\Q$ such that $a,b,q, q-a,q-b,q-a-b \not\in \Z$.
Assume that
\begin{equation}\label{main-cond}
\{sq\}+\{s(q-a-b)\}=\{s(q-a)\}+\{s(q-b)\}, \quad s\in\hat{\Z}^\times.
\end{equation}
Then we have
\begin{equation}\label{main-result}
B(a,b)\cdot {}_3F_2\left(
\begin{matrix}
a,b,q\\
a+b,q+1\end{matrix};1
\right)
\in\ol{\Q}+\ol{\Q}\log\ol{\Q}^\times.
\end{equation}
Here, $\ol{\Q}+\ol{\Q}\log\ol{\Q}^\times$ denotes the $\ol\Q$-linear subspace of $\C$ spanned by
$1$, $2\pi i$ and $\log \alpha$ for all $\alpha\in\ol{\Q}^\times$.
\end{thm}
We note that the action of $\hat{\Z}$ on the subgroup $\frac{1}{N}\Z/\Z$
factors through the finite quotient $(\Z/N\Z)^\times$.
Therefore, taking $N$ so that $a,b,q\in \frac{1}{N}\Z$, the assumption \eqref{main-cond} is verified by taking as $s$ the integers $1,2,\dots, N-1$ prime to $N$. 
When $q=\frac{a+b}{2}$, the assumption is satisfied since $\{x\}+\{1-x\}=1$ for any $x \in \R\setminus \Z$. 
Since $\eqref{main-cond}$ is also written as 
$$\{sq\}+\{s(q-a-b)\}+\{s(a-q)\}+\{s(b-q)\}=2, \quad s\in\hat{\Z}^\times, $$
the condition is symmetric in $\{q,q-a-b,a-q,b-q\}$. 
As well as the assumption, the conclusion of the theorem depends only on the classes of $a$, $b$, $q$ mod $\Z$. This is because of the functional equation of the beta function, e.g. $(a+b)B(a+1,b)=aB(a,b)$, and the contiguous relations among ${}_3F_2$-functions (see \cite[Section 7.3]{a-o}).  
The latter is the reason why we need to consider the values in $\ol\Q+\ol\Q\log\ol\Q^\times$, not only in $\ol\Q\log\ol\Q^\times$. 

By using Thomae's formula (see \cite[Ch. III, 3.2 (1)]{Bailey}) repeatedly, we obtain other expressions of \eqref{1} as follows:  
\begin{align*}
&B(a,b)\cdot {}_3F_2\left(
\begin{matrix}
a,b,q\\
a+b,q+1\end{matrix};1
\right)  \notag
\\&=\frac{q}{ab}\cdot{}_3F_2\left(
\begin{matrix}
1,1,a+b-q\\
a+1,b+1\end{matrix};1
\right)\quad (q>0)
\\&=
\frac{1}{a}\cdot {}_3F_2\left({a,q+1-b,1 \atop a+1,q+1};1\right) \quad (b>0)
\\&=
\frac{1}{b} \cdot {}_3F_2\left({b,q+1-a,1 \atop b+1,q+1};1\right) \quad (a>0)
\\&=
\frac{B(a,b)}{aB(a,q+1-a)}\cdot {}_3F_2\left({a,a,a+b-q \atop a+1,a+b};1\right) \quad (q+1-a>0) 
\\&=
\frac{B(a,b)}{bB(b,q+1-b)}\cdot{}_3F_2\left({b,b,a+b-q \atop b+1,a+b};1\right) \quad (q+1-b>0)
\\&=
\frac{B(a,b)}{qB(a+b-q,q)}\cdot{}_3F_2\left({q+1-a,q+1-b,q \atop q+1,q+1};1\right) \quad (a+b-q>0), 
\end{align*}
where the positivity condition in each line is needed for the convergence.  

\section{First Proof: Hypergeometric Fibrations}\label{proof-sect}

We derive Theorem \ref{main} from the regulator formula in \cite{a-o} for what we call hypergeometric fibrations. In Subsections 3.1 and 3.2, we recall necessary materials from \cite{a-o}. 

\subsection{Hypergeometric fibrations}

Let $X$ be a smooth projective variety over $\ol\Q$ and 
$f\colon X\to\P^1$ be a surjective morphism. 
Let $t$ be the coordinate of $\bA^1 \subset \P^1$, $X_t$ be the general fibre of $f$, and 
$H^*(X_t,\Q)$ denote the Betti cohomology of $X_t(\C)$. 
Let $R_0$ be a semi-simple finite-dimensional $\Q$-algebra and $e_0\colon R_0 \to E_0$ be a projection onto a number field.  
We say that $f$ is a {\em hypergeometric fibration} with respect to $e_0$ if the following conditions are satisfied: 
\begin{enumerate}
\item[(a)]
$f$ is smooth over $\P^1\setminus\{0,1,\infty\}$.
\item[(b)]
After restricting to a non-empty Zariski open subset of $\P^1$, there is a ring homomorphism $R_0 \to \End (R^1f_*\Q)$, 
such that $$\dim_{E_0}e_0H^1(X_t,\Q)=2,$$
where we put $e_0M=E_0\ot_{R_0,e_0}M$ for an $R_0$-module $M$. 
\item[(c)]
The local monodromy $T_1$ at $t=1$ on $e_0H^1(X_t,\Q)$ 
is unipotent, and 
$$\mathrm{rank}(\log T_1)=[E_0:\Q].$$
\end{enumerate}

For each embedding $\chi \colon E_0\hra\ol\Q$,  
let $(R^1f_*\ol\Q)^\chi$ denote the $\chi$-part which is by definition the
subspace on which $g\in E_0$ acts as multiplication by $\chi(g)$. 
Let $T_p$ be the local monodromy at $t=p \in \{0,\infty\}$ 
on the rank-two $\ol\Q$-local system $(R^1f_*\ol\Q)^{\chi}$. 
Then the eigenvalues of $T_0$ (resp. $T_\infty$) are written as
$e^{2\pi i\alpha^\chi_1}$, $e^{2\pi i\alpha^\chi_2}$  
(resp. $e^{2\pi i\beta^\chi_1}$, $e^{2\pi i\beta^\chi_2}$),
where $\a_i^\chi, \beta_i^\chi \in \Q$. 

\subsection{Regulator formula}
Now, take a positive integer $l$, and let $\pi\colon\P^1\to \P^1$ be the map given by $\pi(t)=t^l$.
We consider the variation of Hodge--de Rham structures 
$$\cM^{(l)}:=\pi_*\Q\ot R^1f_*\Q,$$
and the cohomology groups
\begin{equation*}\label{hm-def}
H^{(l)}:=H^1(\P^1,j_*\cM^{(l)}),\quad 
M^{(l)}:=H^1(\P^1\setminus\{0,1,\infty\},\cM^{(l)}), 
\end{equation*}
where $j\colon \P^1\setminus\{0,1,\infty\} \hookrightarrow\P^1$ is the immersion. 
Then, there is an exact sequence of mixed Hodge--de Rham structures (see \cite[Section 4.2]{a-o})
\begin{equation}\label{ext0}
0\lra H^{(l)}\lra M^{(l)}\lra C^{(l)}\lra 0, 
\end{equation}
where 
$$C^{(l)}:=\bigoplus_{p=0,1,\infty}C_p^{(l)}, \quad C_p^{(l)}:=(R^1j_*\cM^{(l)})_p
\cong\Coker\left[T_p-1\colon \cM_p^{(l)}\to\cM_p^{(l)}\right].$$
We recall that a {\em Hodge--de Rham structure} 
is a quadruple $H=(H_B,H_\dR, F^\bullet,\iota)$ of 
finite-dimensional vector spaces over $\ol\Q$,  a descending filtration of $H_\dR$, and a comparison isomorphism $H_{B,\C} \os\cong\to H_{\dR,\C}$ satisfying standard properties (see \cite[Section 2.1]{a-o}).   

Since $\Aut(\pi)=\mu_l$, the group ring $R:=R_0[\mu_l]$ acts on the exact sequence \eqref{ext0}. 
Let $e \colon R\to E$ be a projection onto a number field $E$ which extends $e_0\colon R_0 \to E_0$. 
For each embedding $\chi \colon E \hra \ol\Q$, define $k^\chi\in\Z/l\Z$ by 
$\chi(\zeta)=\zeta^{k^\chi}$ for $\zeta\in\mu_l$, and put $\kappa^\chi =k^\chi/l \in \Q/\Z$. 
We write the restriction of $\chi$ to $E_0$ by the same letter and then $\a_i^\chi$, $\b_i^\chi \in \Q/\Z$ are defined as above. 

Now we suppose:
\begin{equation}\label{cond abk}
\k^\chi+\a_1^\chi, \ \k^\chi+\a_2^\chi, \ \k^\chi-\b_1^\chi, \ \k^\chi-\b_2^\chi\ \not\in\Z.
\end{equation}
Then, it is not hard to show that 
$eC^{(l)}=eC_1^{(l)}$, $\dim_EeC^{(l)}=\dim_EeH^{(l)}=1$, and that 
$eC^{(l)}$ (resp. $eH^{(l)}$) is a pure Hodge structure of type $(2,2)$ (resp. of weight $2$) (see \cite[Section 4.3]{a-o}). 
By an identification $eC^{(l)}=E(-2)$, we obtain from \eqref{ext0} an exact sequence
$$0 \lra eH^{(l)}(2) \lra eM(2) \lra E \lra 0.$$
Throughout the remaining of this section, write for brevity $H=eH^{(l)}$. 
We have the connecting homomorphism
$$\rho\colon E\lra \Ext^1(\Q, H(2))$$
to the Yoneda extension group of mixed Hodge--de Rham structures. 
Denote by $H^\chi=(H_B^\chi,H_\dR^\chi,F^\bullet,\iota)$ the $\chi$-part of $H$, i.e.
the subspace on which each $\sigma\in G$ acts as multiplication by $\chi(\sigma)$.
The period $\mathrm{Per}(H^\chi) \in \C^\times/\ol\Q^\times$ in the sense of Deligne \cite{deligne} is defined by $\iota(H_{\dR}^\chi)=\mathrm{Per}(H^\chi)H_B^\chi$. 
Choose a $\ol\Q$-basis $\eta$ of $(eH_\dR^{(l)})^\chi$, 
and let $i_\eta$ be the composition of the following maps:
\begin{align*}
\Ext^1(\Q,H(2))
&\os{\cong}{\lra} H_{\dR,\C}/\left(F^2H_\dR+\iota(H_B(2))\right)\\
&\lra H_{\dR,\C}^\chi/\left(F^2H^\chi_\dR+\iota (H_B^\chi(2))\right)\\
&\os{\cong}{\lra} \C/\left(\ol\Q\delta_\chi+\ol\Q\mathrm{Per}(H^{\ol\chi})\right),
\end{align*}
where we put $\delta_\chi=0$ or $1$ depending on whether $F^2H_\dR^\chi=0$ or not, 
and $\ol\chi$ is the complex conjugate of $\chi$.  
Here, the first map is the Carlson isomorphism, the second map is the projection to the $\chi$-part, and the last isomorphism sends $\eta$ to $1$. Note that 
$\mathrm{Per}(H^\chi)\cdot\mathrm{Per}(H^{\ol\chi}) \in(2\pi i)^2\ol\Q$. 
Put $\rho^\chi = i_\eta\circ \rho$. 

Now, our regulator formula is the following. 
\begin{thm}[\cite{a-o}, Theorem 4.7]\label{a-o-thm}
Let the notation and assumption be as above. 
Then there exist $c_1, c_2 \in\ol\Q$, $c_2\ne 0$, such that 
\begin{equation}\label{main-regterm}
\rho^{\ol\chi}(1)
=c_1+c_2
B(\alpha^\chi_1+\beta^\chi_1,\alpha^\chi_1+\beta^\chi_2)~
{}_3F_2\left(\begin{matrix}\alpha^\chi_1+\beta^\chi_1,\alpha^\chi_1+\beta^\chi_2,
\k^\chi+\alpha^\chi_1\\
2\alpha^\chi_1+\beta^\chi_1+\beta^\chi_2,\k^\chi+\alpha^\chi_1+1\end{matrix};1
\right)
\end{equation}
in $\C/\left(\ol\Q\delta_\chi +\ol\Q\mathrm{Per}(H^\chi)\right)$. 
\end{thm}

The period formula \cite[Theorem 4.5]{a-o} (see also \cite[Theorem 5.4]{a-o-1}) reads
\begin{equation*}
\mathrm{Per}(H^\chi) \sim_{\ol\Q^\times} 
2\pi i\cdot \Gamma\left({\k^\chi+\a_1^\chi, \k^\chi+\a_2^\chi\atop \k^\chi-\b_1^\chi,\k^\chi-\b_2^\chi}\right).
\end{equation*}
Note that the second term of the right-hand side of \eqref{main-regterm} is written as $cB(a,b)\cdot {}_3F_2\left({a,b,q\atop a+b,q+1};1\right)$ by letting 
$$a=\a_1^\chi+\b_1^\chi, \ b=\a_1^\chi+\b_2^\chi, \ q=\k^\chi+\a_1^\chi.$$
Then, since $\a_1^\chi+\a_2^\chi+\b_1^\chi+\b_2^\chi \in\Z$, we have
$$
\Gamma\left({\k^\chi+\a_1^\chi, \k^\chi+\a_2^\chi\atop \k^\chi-\b_1^\chi,\k^\chi-\b_2^\chi}\right)
\sim_{\Q^\times} \Gamma\left({q, q-a-b\atop q-a,q-b}\right).
$$
By Koblitz--Ogus \cite[p.344, Theorem]{deligne}, the condition \eqref{main-cond} implies that
$$\Gamma\left({sq, s(q-a-b)\atop s(q-a),s(q-b)}\right) \in\ol\Q^\times$$
for any $s \in \hat\Z^\times$, hence $\mathrm{Per}(H^\chi) \in \ol\Q(2\pi i)$ for any $\chi$. 

\subsection{Algebraic cycles}
The connecting homomorphism $\rho^\chi$ 
is related with Beilinson's regulator map 
from the motivic cohomology group.
Consider the diagram
\[
\xymatrix{
X^{(l)}\ar[r]^i\ar[rd]_{f^{(l)}}&
X_l\ar[r]\ar[d]\ar@{}[rd]|{\square}&X\ar[d]^f\\
&\P^1\ar[r]^\pi&\P^1
}
\]
where $X_l:=X \times_{\P^1, f} \P^1$ and $i$ is the desingularization.
Put $D^{(l)}:=(\pi\circ f^{(l)})^{-1}(1)$, a union of $l$ copies of  the fiber $f^{-1}(1)$.
There are canonical isomorphisms 
\begin{align*}
&C_1^{(l)}\cong H_1(D^{(l)},\Q)(-2), \\
&H^{(l)}
\cong \ker\left[H^2(X^{(l)},\Q)/N^1_\mathrm{fib}(X^{(l)}) \to H^2(X^{(l)}_t,\Q)\right],
\end{align*}
where $N^1_\mathrm{fib}(X^{(l)})\subset $ denotes 
the classes of fibral divisors for $f^{(l)}$.
Then we have a commutative diagram \cite[Proposition 4.8]{a-o}
\[
\xymatrix{
H^3_{\cM,D^{(l)}}(X^{(l)},\Q(2))\ar[r]\ar[d]_{\reg_{D^{(l)}}}
&H^3_\cM(X^{(l)},\Q(2))\ar[d]^{\reg_{X^{(l)}}}\\
H_1^B(D^{(l)},\Q)\ar[r]
&\Ext^1(\Q, (H^2(X^{(l)},\Q)/N^1_\mathrm{fib})(2)),
}
\]
where the vertical maps are the regulators and the lower horizontal map is the connecting homomorphism induced by \eqref{ext0}. 

\begin{prop}\label{a-o-lem}
Suppose that $H=eH^{(l)}$ is a Hodge structure of type $(1,1)$. 
Then we have \/
$\Image(\rho^{\chi})\subset  \ol{\Q}\log\ol{\Q}^\times/\ol\Q \cdot 2\pi i$ 
for any embedding $\chi\colon E \hra \ol\Q$, 
\end{prop}

\begin{proof}
Recall that the target of $\rho^{\chi}$ is $\C/\ol\Q\cdot 2\pi i$ by the assumption and the remark after Theorem \ref{a-o-thm}. 
Since $\reg_{D^{(l)}}$ is surjective, it is suffices to consider the image of 
$H^3_{\cM,D^{(l)}}(X^{(l)},\Q(2))$ under $\rho^\chi$.

Let $N^r(X^{(l)})\subset H^{2r}(X^{(l)},\Q)$ be the subspace generated by
algebraic cycles of codimension $r$. Note that it is generated by cycles defined 
over $\ol\Q$.
By the assumption and Lefschetz's theorem (i.e. the Hodge conjecture for $H^2$), 
we have $eH^{(l)}\subset N^1(X^{(l)})$.
The intersection pairing $N^1(X^{(l)})\ot N^{\dim X-1}(X^{(l)})\to \Q$ is non-degenerate 
by the non-degeneracy of the pairing on the N\'eron--Severi group.
This implies that there is a smooth projective curve $C$  (not necessarily connected) 
and a morphism $C\to X^{(l)}$ such that the image of $C$ intersects properly with $D^{(l)}$, 
the pull-back $H^2(X^{(l)}) \to H^2(C)$ annihilates $N^1_\mathrm{fib}(X^{(l)})$, 
and the composition 
$$H \to H^2(X^{(l)})/N^1_\mathrm{fib}(X^{(l)},\Q)\to H^2(C,\Q)$$ 
is injective.
Then we have a commutative diagram
$$\xymatrix{
& H^3_{\cM,D^{(l)}}(X^{(l)},\Q(2))\ar[r]\ar[d]_{\rho\circ\reg_{D^{(l)}}}
&H^3_{\cM,D^{(l)}\cap C}(C,\Q(2))\ar[d]^{\rho\circ \reg_{D^{(l)}\cap C}}\\
\Ext^1(\Q,H(2))\ar[r]&\Ext^1(\Q,H^2(X^{(l)},\Q)/N^1_{\mathrm{fib}}(2))\ar[r]&\Ext^1(\Q,H^2(C,\Q(2))), 
}$$
where the composite of the lower horizontal maps is injective. 

Since $H^2(C,\Q(2))\cong\bigoplus_{C^{(0)}} \Q(1)$, 
where $C^{(0)}$ denotes the set of connected components of $C$, 
the map $\reg_{D^{(l)}\cap C}$ is canonically identified with the logarithm map
$$\bigoplus_{D^{(l)}\cap C} \Q \ot_\Z \ol\Q^\times  \lra 
\bigoplus_{C^{(0)}}\C/\Q\cdot 2\pi i.$$
Since $H\cong \Q(-1)^{\oplus[E\colon \Q]}$, we have $\Ext^1(\Q,H(2))\cong (\C/\Q\cdot2\pi i)^{\oplus[E\colon \Q]}$, 
and this injects to $\bigoplus_{C^{(0)}}\C/\Q\cdot 2\pi i$. 
Hence, taking the $\chi$-part, the lemma follows. 
\end{proof}

\begin{prop}\label{a-o-cor}
Let the notation and assumption be as in Theorem \ref{a-o-thm}.
Then $H=eH^{(l)}$ is a Hodge structure of type $(1,1)$ if and only if
\begin{equation}\label{cond chi}
\left\{\k^\chi+\alpha^\chi_1\right\}
+\left\{\k^\chi+\alpha_2^\chi\right\}
=\left\{\k^\chi-\beta^\chi_1\right\}
+\left\{\k^\chi-\beta_2^\chi\right\}
\end{equation}
for any $\chi$. 
\end{prop}

\begin{proof}
The first assertion follows from an explicit formula \cite{asakura-fresan}
of the Hodge type of $H$, 
which is proven using the Riemann--Roch--Hirzebruch theorem. 
For the hypergeometric fibration of Gauss type, 
which will be used below to prove Theorem \ref{main},  
this is computed in \cite[Theorem 5.4]{a-o-1} 
(the situation in loc. cit.  is more restricted but the same argument works in general). 
\end{proof}

\subsection{Hypergeometric fibration of Gauss type}
We finish the proof of Theorem \ref{main}. 
In view of Theorem \ref{a-o-thm}, Proposition \ref{a-o-lem} and Proposition \ref{a-o-cor}, 
it suffices to find a fibration $f$ such that 
$$a\equiv \alpha^\chi_1+\beta^\chi_1,\quad
b\equiv\alpha^\chi_1+\beta^\chi_2,\quad
q\equiv\alpha^\chi_1+\k^\chi \pmod\Z$$
for some $\chi$, and the condition \eqref{cond chi} is satisfied for any $\chi$. 
Note that the condition \eqref{main-cond} implies \eqref{cond chi} for any $\chi$. 
The non-integrality condition of Theorem \ref{main} is equivalent to  
$\a_i^\chi+\b_j^\chi \not\in\Z$ ($i, j \in \{1,2\}$) and \eqref{cond abk}. 

We may and do suppose $0<a, b, q<1$. Let $N$ be the smallest positive integer such that $A:=Na$, $B:=Nb \in \Z$. 
Consider the hypergeometric fibration of Gauss type
$$y^N=x^A(1-x)^B(1-tx)^{N-B}.$$
This is a principal example of hypergeometric fibrations studied in detail in \cite{a-o-1} and in \cite[Section 3.2]{a-o}. 
Let $\z_N\in \mu_N$ be a primitive $N$th root of unity. 
The group ring $R_0:=\Q[\mu_N]$ acts on $X$ by letting $\z_N$ act by $y \mapsto \z_N^{-1}y$. 
 Let 
$e_0\colon R_0\to E_0=\Q(\mu_N)$ 
be the natural projection. For each embedding $\chi \colon E_0 \hookrightarrow \C$, such that $\chi(\z_N)=\z_N^s$, we have
$$\a_1^\chi=0, \ \a_2^\chi=\{1-\b_1^\chi-\b_2^\chi\}, \ \b_1^\chi=\left\{\frac{sA}{N}\right\}, \ \b_2^\chi=\left\{\frac{sB}{N}\right\}.$$ 
Hence we have $a=\a_1^\chi+\b_1^\chi$, $b=\a_1^\chi+\b_2^\chi$ for the trivial embedding $\chi$ (i.e. $s=1$). 
Let $l$ be the smallest positive integer such that $Q:=lq\in\Z$ and let 
$e\colon R=R_0[\mu_l]\to E=E_0(\mu_l)$ 
be an extension of $e_0$ given by $\z_l\mapsto \z_l^Q$. Then, for the trivial embedding $\chi$ of $E$, we have $q=\k^\chi$. 
Hence this fibration has the desired property and Theorem \ref{main} is proved. 

\section{Second Proof: Fermat Surfaces}\label{fermat-sect}
We give the second proof of Theorem \ref{main}, by studying extensions of mixed Hodge--de Rham structures coming from Fermat surfaces. 
Throughout this section, we assume 
$a,b,q\in\Q$ and $a,b,q, q-a,q-b,q-a-b \not\in \Z$.

\subsection{Integral representation}
Let us begin with the integral representation of ${}_3F_2$-function (cf. \cite{slater}):
\begin{align*}
&B(\alpha_1,\beta_1-\alpha_1)B(\alpha_2,\beta_2-\alpha_2){}_3F_2\left(
\begin{matrix}
\alpha_1,\alpha_2,\alpha_3\\
\beta_1,\beta_2\end{matrix};z
\right)
\\&=\int^1_0\int^1_0
t_1^{\alpha_1-1}t_2^{\alpha_2-1}
(1-t_1)^{\beta_1-\alpha_1-1}
(1-t_2)^{\beta_2-\alpha_2-1}(1-zt_1t_2)^{-\alpha_3}\, dt_1dt_2. 
\end{align*}
Set $(\alpha_1,\alpha_2,\alpha_3,\beta_1,\beta_2)=(a,q,b,a+b,q+1)$.
By the change of variables $x=t_1$, $y=(1-t_1)/(1-zt_1t_2)$, we obtain 
\begin{equation*}
B(a,b){}_3F_2\left(
\begin{matrix}
a,q,b\\
a+b,q+1\end{matrix};z
\right)=qz^{-q}\int_{E_z}
x^{a-q-1}y^{b-q-1}(x+y-1)^{q-1}\,dxdy, 
\end{equation*}
where $E_z$ is the domain in the $xy$-plane corresponding to 
$\{(t_1,t_2)\mid 0\leq t_1,t_2\leq 1\}$.
Suppose that $a,b,q\in \frac{1}{N}\Z$.
We take new variables $u,v,w$ such that
\[
u^N=x,\quad v^N=y,\quad w^N=u^N+v^N-1=x+y-1.
\] 
Then we have
\begin{equation*}\label{f1}
B(a,b){}_3F_2\left(
\begin{matrix}
a,q,b\\
a+b,q+1\end{matrix};z
\right)=N^2qz^{-q}\int_{\Delta_z}
u^{N(a-q)-1}v^{N(b-q)-1}w^{Nq-N}\,dudv, 
\end{equation*}
where $\Delta_z$ is an arbitrary domain in the $uv$-plane which corresponds to 
$E_z$.
Substitute $z=1$ and choose the domain as
\begin{equation}\label{f-delta}
\Delta=\Delta_1:=\left\{(u,v)\in \R^2\mid 0\leq u,v\leq 1,~1\leq u^N+v^N\right\}. 
\end{equation}
Then we obtain
\begin{equation}\label{f2}
B(a,b){}_3F_2\left(
\begin{matrix}
a,q,b\\
a+b,q+1\end{matrix};1
\right)=N^2q\int_{\Delta}
u^{N(a-q)-1}v^{N(b-q)-1}w^{Nq-N}\,dudv.
\end{equation}
We shall give a motivic interpretation of this integral.

\subsection{Fermat surface}
The differential form
\begin{equation*}
\omega:=u^{N(a-q)-1}v^{N(b-q)-1}w^{Nq-N}\,dudv
\end{equation*}
defines a de Rham cohomology class $\eta\in H_\dR^2(S)$
of the Fermat surface over $\ol\Q$
\[
S:u^N+v^N-1=w^N.
\]
Let the group $G:=\mu_N^3$ act on $S$ by 
$\sigma(u,v,w)=(\zeta_1u,\zeta_2v,\zeta_3w)$ for $\sigma=(\zeta_1,\zeta_2,\zeta_3)\in G$.
This time, we first fix a $\Q$-algebra homomorphism 
$$\chi\colon\Q[G]\to \ol\Q, 
\quad \chi(\zeta_1,\zeta_2,\zeta_3)=\zeta_1^{N(a-q)}\zeta_2^{N(b-q)}\zeta_3^{Nq}.$$
Let $E$ be the coimage of $\chi$, and $e\in\Q[G]$ be the corresponding idempotent, 
i.e. $e^2=e$ and $e\Q[G]\cong E$.
Let $D$ be the union of curves on $S$ defined by 
$$(u^N-1)(v^N-1)w=0,$$
which is stable under the $G$-action.  

\begin{lem}\label{lem-f-1}
We have 
$$\dim_EeH_1(D,\Q)=1, \quad
\dim_EeH_2(D,\Q)=0,  \quad
\dim_E eH_2(S,\Q)=1.$$
Moreover, $eH_1(D,\Q)$ is a Hodge--de Rham structure of type $(0,0)$.
\end{lem}
\begin{proof}
The former statement is an easy exercise. To see the latter, let $\pi\colon\wt{D}\to D$ be the
normalization, $\Sigma$ be the set of singular points of $D$ and put $\wt\Sigma=\pi^{-1}(\Sigma)$.
Then there is an exact sequence
$$H^0(\wt{D})\lra \Q_{\wt\Sigma}/\Q_\Sigma \lra H^1(D)\lra H^1(\wt{D})\lra 0, $$
where $\Q_\Sigma:=\mathrm{Maps}(\Sigma,\Q)$. 
This remains exact after applying $e$. 
Since $D$ is a union of rational curves and the Fermat curve of  degree $N$, 
and $(1,1,\z_3)$ acts trivially on the latter, we have $eH^1(\wt{D})=0$ by the assumption $q\not\in\Z$. 
Hence the assertion follows.
\end{proof}

Put $H=eH_2(S)$, a Hodge--de Rham structure of type $(0,-2)$, $(-1,-1)$, $(-2,0)$. 

\begin{prop}\label{prop-f-4}
The Hodge type of $H$ is $(-1,-1)$ if and only if 
$$\{sq\}+\{s(q-a-b)\}=\{s(q-a)\}+\{s(q-b)\}$$
holds for any $s\in \hat{\Z}^\times$.
\end{prop}
\begin{proof}
As is well-known, the cohomology $eH^2(S,\Q)$ is generated
by the classes of rational $2$-forms
$$\eta_s:=u^{N\{s(a-q)\}-1}v^{N\{s(b-q)\}-1}w^{N\{sq\}-N}\,dudv,\quad s\in (\Z/N\Z)^\times,$$
and $\eta_s$ belongs to the Hodge $(p_s,2-p_s)$-component, 
where
$$p_s:=\{s(q-a)\}+\{s(q-b)\}+\{-sq\}-\{s(q-a-b)\}.$$
Since $eH_2(S,\Q)$ has the Hodge type $(-1,-1)$ if and only if 
$eH^2(S,\Q)$ has the Hodge type $(1,1)$, the assertion follows. 
\end{proof}

\subsection{Extension of mixed Hodge--de Rham structures}
By Lemma \ref{lem-f-1}, we have an exact sequence
\begin{equation*}\label{f3}
0\lra H \lra eH_2(S,D;\Q) \os\partial\lra  e H_1(D,\Q) \lra 0
\end{equation*}
of mixed Hodge--de Rham structures.
As before, we have the connecting map
\begin{equation*}
\rho\colon eH_1^B(D,\Q)\lra \Ext^1(\Q,H)
\end{equation*}
to the Yoneda extension group of mixed Hodge--de Rham structures. 
Regarding $\eta \in H_\dR^2(S)^\chi$ as an element of $H_2^\dR(S)^\chi$ by the Poincar\'e duality, we obtain as before a map 
$$
i_\eta\colon \Ext^1(\Q,H)
\lra \C/\left(\ol\Q\delta_\chi+\ol\Q\mathrm{Per}(H^{\ol\chi})\right),
$$
where we put $\delta_\chi=0$ or $1$ depending on whether $F^0H_\dR^\chi=0$ or not, 
and $\ol\chi$ is the complex conjugate of $\chi$.  
One easily sees that the cycle $\Delta$ given in \eqref{f-delta} defines a homology cycle in $H_2^B(S,D;\Z)$.
Let $\delta:=\partial(e\Delta) \in eH^B_1(D,\Q)$ be the boundary.

\begin{prop}\label{prop-f-2}
Write $\rho^\chi=i_\eta\circ \rho$. Then we have
\begin{equation*}\label{f5}
\rho^\chi(\delta)=c+
\frac{1}{N^2q}
B(a,b)\cdot {}_3F_2\left(
\begin{matrix}
a,b,q\\
a+b,q+1\end{matrix};1
\right)
\end{equation*}
for some $c \in \ol\Q$ in 
$\C/\left(\ol\Q\delta_\chi+\ol\Q\mathrm{Per}(H^{\ol\chi})\right)$. 
\end{prop}

\begin{proof}
Consider the exact sequence
\[
0\lra eH^1_\dR(D)\os{h}{\lra} eH^2_\dR(S,D)\lra eH^2_\dR(S)\lra 0. 
\]
Let 
$\wt{\eta}\in eH^2_\dR(S,D)^\chi$ be the unique lifting of $\eta\in eH^2_\dR(S)^\chi$ contained in $F^1$.
Let 
$$\langle\ , \ \rangle\colon eH_2^B(S,D;\Q)\ot eH^2_\dR(S,D)\to\C, \quad 
\langle\ , \ \rangle\colon eH_1^B(D,\Q)\ot eH^1_\dR(D)\to\C$$
be the natural pairings. By Lemma \ref{lem-f-1}, the latter maps to $\ol\Q$. 
As is easily seen from the definition, we have 
$\rho^\chi(\delta)=\langle e\Delta,\wt{\eta}\rangle=\langle\Delta,\wt{\eta}\rangle$
in $\C/\left(\ol\Q\delta_\chi+\ol\Q\mathrm{Per}(H^{\ol\chi})\right)$. 
For an arbitrary lifting $\wt\eta'\in eH^2_\dR(S,D)^\chi$ of $\eta$, 
there exists $\xi \in eH^1_\dR(D)$ such that $h(\xi)=\wt\eta-\wt\eta'$.
Then we have 
$c:=\langle \Delta, \wt\eta-\wt\eta'\rangle=\langle \delta,\xi\rangle \in \ol\Q$. 
As $\wt\eta'$, we can choose the one 
represented by the \v{C}ech cocycle 
$$(0,0,\omega)\in \C_{\wt\Sigma}/\C_\Sigma \op \cA^1(\wt D)\op \cA^2(S),$$
where $\cA^q(M)$ denotes the space of smooth differential $q$-forms on $M$ with $\C$-coefficients (see \cite[Section A.1]{a-o-1}). 
Then, by \cite[Theorem A.3]{a-o-1}, we have 
$\langle\Delta,\wt\eta'\rangle=\int_\Delta \omega$ 
in $\C/\left(\ol\Q\delta_\chi+\ol\Q\mathrm{Per}(H^{\ol\chi})\right)$. 
Hence the proposition follows by \eqref{f2}. 
\end{proof}

Now, by applying a similar argument as in Proposition \ref{a-o-lem}, Theorem \ref{main} follows 
from Proposition \ref{prop-f-4} and Proposition \ref{prop-f-2}. 

\section{Open Problems}\label{problem-sect}

First, contrary to Wilson's formula, our result does not generally give an explicit formula expressing the value of ${}_3F_2$ in terms of logarithms. 

\begin{prob}
Give an explicit description of \eqref{main-result} in terms of logarithms.
\end{prob}

In the study of Hodge cycles on Fermat surfaces, Shioda \cite{shioda} gave a conjecture which determines those $(a,b,q)$ satisfying the condition \eqref{main-cond}, 
and  it was proved by Aoki \cite{aoki}.   
Up to permutations of $\{q, q-a-b, a-q, b-q\}$, those are (modulo $\Z^3$):  
\begin{equation*}
(a,b,q)=\begin{cases}
(\a,\b,\frac{\a+\b}{2}),\\
(\a,\a+\frac{1}{2},2\a),\\
(2\a+\frac{1}{3},2\a+\frac{2}{3},3\a),\\
(3\a+\frac{1}{4},3\a+\frac{3}{4},4\a),
\end{cases} \quad \a,\b\in\Q, 
\end{equation*}
except for a finite number of exceptional cases (see \cite[Appendix]{terasoma} for the list). 
Expanding the method in Section 3, Yabu \cite{yabu} computes several examples including: 
\begin{equation}\label{ex yabu}
\begin{split}
2\pi \cdot {}_3F_2\left({\frac{1}{6},\frac{5}{6},\frac{1}{3} \atop 1, \frac{4}{3}};1\right)
&=
2^\frac{1}{3}\sqrt{3} \cdot \log \alpha- 2^\frac{7}{3} \cdot \mathrm{Cot}^{-1}\beta,\\ 
3\pi \cdot {}_3F_2\left({\frac{1}{6},\frac{5}{6},\frac{2}{3} \atop 1, \frac{5}{3}};1\right)
&=
2^\frac{2}{3}\sqrt{3}\cdot \log \alpha+2^\frac{5}{3} \cdot \mathrm{Cot}^{-1}\beta,
\end{split}
\end{equation}
with
$$\alpha=\frac{(2^\frac{2}{3}-1)^2+(2^\frac{2}{3}+\sqrt{3})^2}{(2^\frac{2}{3}-1)^2+(2^\frac{2}{3}-\sqrt{3})^2}, \quad 
\beta=\frac{3+2^\frac{1}{3}+2^\frac{2}{3}\cdot3}{3}.$$
Recently, expanding the method in Section 4, the third author \cite{terasoma} solved the problem except for the exceptional cases. 

Finally, as we have seen, \eqref{main-cond} is a necessary and sufficient condition
for that $eH^{(l)}$ in (Section \ref{proof-sect}) or $eH$ in (Section \ref{fermat-sect}) 
is isomorphic to the Tate object $E\ot \Q(-1)$ or $E\ot \Q(1)$, respectively. 
If this is not the case, there is no reason for and it seems rather weird that the regulator value \eqref{main-result} is expressed in terms of logarithms of algebraic numbers. 
Hence it would be fair to raise the following conjecture. 

\begin{conj}
Under the same assumption as in Theorem \ref{main}, \eqref{main-cond} is a necessary condition for \eqref{main-result}. 
\end{conj}

\section*{Acknowledgements}
This work is supported by JSPS Grant-in-Aid for Scientific Research: 15K04769, 25400007 and 15H02048.

\end{document}